\documentclass{jfp1}
\usepackage{proof}
\usepackage{amssymb}
\usepackage{txfonts}

\def%
\D{{\cal D}}
\def%
\R{{\cal R}}
\def\MRL{\mbox{MRL}}

\def\LMRL{\mbox{LMRL}}

\def\prim{a}
\def\fneg{f}
\def\limp{\supset}
\def\subst#1#2#3{#3[#2/#1]}
\def%
\mrl#1#2#3%
{#2\,{\land_#1}\,#3}
\def%
\mrlq#1#2%
{{\forall_#1}(#2)}

\def\interp#1#2{[#2]_{#1}}

\def\setcomp#1{{\overline#1}}
\def\jimp{\Rightarrow}
\def\fullset{\setcomp{\emptyset}}
\def\ncut#1{\mbox{\it{#1}-cut}}
\def\ncutconj#1{\mbox{\it{#1}-cut-conj}}
\def\ncutdisj#1{\mbox{\it{#1}-cut-disj}}
\def\tpjg{\vdash}

\def\height#1{\mbox{\it ht}(#1)}

\title
[Journal of Functional Programming]
{Multirole Logic\break{\small(Extend Abstract)}}

\author
[Hongwei Xi and Hanwen Wu]
{
{Hongwei Xi and Hanwen Wu} \\
Boston University, Boston, MA 02215, USA\\
\email{hwxi@cs.bu.edu, hwwu@cs.bu.edu}
}

\jdate{April, 2016}

\begin{document}

\label{firstpage}

\maketitle

\begin%
{abstract}
We identify multirole logic as a new form of logic in which
conjunction/disjunction is interpreted as an ultrafilter on the
power set of some underlying set (of roles) and the notion of negation
is generalized to endomorphisms on this underlying set.  We formalize
both multirole logic ($\MRL$) and linear multirole logic ($\LMRL$) as
natural generalizations of classical logic (CL) and classical linear
logic (CLL), respectively, and also present a filter-based
interpretation for intuitionism in multirole logic. Among various
meta-properties established for $\MRL$ and $\LMRL$, we obtain one
named multiparty cut-elimination stating that every cut involving one
or more sequents (as a generalization of a (binary) cut involving
exactly two sequents) can be eliminated, thus extending the celebrated
result of cut-elimination by Gentzen.
\end{abstract}

\tableofcontents

\section
{Introduction}
While the first and foremost inspiration for multirole logic came to
us during a study on multiparty session types in distributed
programming~\cite{MTLC-i-sessiontype,MTLC-g-sessiontype}, it seems
natural in retrospective to introduce multirole logic by exploring (in
terms of a notion referred to as role-based interpretation) the
well-known duality between conjunction and disjunction in classical
logic. For instance, in a two-sided presentation of the classical
sequent calculus (LK), we have the following rules for conjunction and
disjunction:
\def\Ggamma{\underline{A}}
\def\Ddelta{\underline{B}}
$$
\begin%
{array}{c}
\infer%
[\hbox{\bf(conj-r)\hss}]
{\Ggamma\tpjg\Ddelta, A\land B}
{\Ggamma\tpjg\Ddelta, A & \Ggamma\tpjg\Ddelta, B}
\\[6pt]
\infer%
[\hbox{\bf(conj-l-1)\hss}]
{\Ggamma, A\land B\tpjg\Ddelta}{\Ggamma, A\tpjg\Ddelta}
\kern18pt
\infer%
[\hbox{\bf(conj-l-2)\hss}]
{\Ggamma, A\land B\tpjg\Ddelta}{\Ggamma, B\tpjg\Ddelta}
\\[6pt]
\infer%
[\hbox{\bf(disj-l)\hss}]
{\Ggamma, A\lor B\tpjg\Ddelta}
{\Ggamma, A\tpjg\Ddelta & \Ggamma, B\tpjg\Ddelta}
\\[6pt]
\infer%
[\hbox{\bf(disj-r-1)\hss}]
{\Ggamma\tpjg\Ddelta, A\lor B}{\Ggamma\tpjg\Ddelta, A}
\kern18pt
\infer%
[\hbox{\bf(disj-r-2)\hss}]
{\Ggamma\tpjg\Ddelta, A\lor B}{\Ggamma\tpjg\Ddelta, B}
\\
\end{array}
$$
where $\Ggamma$ and $\Ddelta$ range over sequents (that are
essentially sequences of formulas).  One possibility to explain this
duality is to think of the availability of two roles $0$ and $1$ such
that the left side of a sequent judgment (of the form
$\Ggamma\tpjg\Ddelta$) plays role $1$ while the right side does role
$0$. In addition, there are two logical connectives $\land_0$ and
$\land_1$; $\land_r$ is given a conjunction-like interpretation by the
side playing role $r$ and disjunction-like interpretation by the other
side playing role $1-r$, where $r$ ranges over $0$ and $1$. With this
explanation, it seems entirely natural for us to introduce more roles
into classical logic.

\def\RC{\mbox{\underline{\rm R}}}
Multirole logic is parameterized over a chosen underlying set of
roles, which may be infinite, and we use $\fullset$ to refer to this
set. Given a subset $R$ of $\fullset$, we use $\setcomp{R}$ for the
complement of $R$ in $\fullset$. Also, we use $R_1\uplus R_2$ for the
disjoint union of $R_1$ and $R_2$ (where $R_1$ and $R_2$ are assumed
to be disjoint).

For the moment, let us assume that $\fullset$ consists all of the
natural numbers less than $N$ for some given $N\geq 2$.  Intuitively,
a conjunctive multirole logic is one in which there is a logical
connective $\land_r$ for each $r\in\fullset$ such that $\land_r$ is
given a conjunction-like interpretation by a side playing role $r$ and
a disjunction-like interpretation otherwise.  If we think of the
universal quantifier $\forall$ as an infinite form of conjunction,
then what is said about $\land$ can be readily applied to $\forall$ as
well. In fact, additive, multiplicative, and exponential connectives
in linear logic~\cite{LinearLogic} can all be treated in a similar
manner.  Dually, a disjunctive multirole logic can be formulated (by
giving $\land_r$ a disjunction-like interpretation if the side plays
the role $r$ and a conjunction-like interpretation otherwise). For
brevity, we primarily focus on conjunctive multirole logic in this
paper.

Given a formula $A$ and a set $R$ of roles, we write $\interp{R}{A}$
for an i-formula, which is some sort of interpretation of $A$ based on
$R$. For instance, the interpretation of $\land_r$ based on $R$ is
conjunction-like if $r\in R$ holds, and it is disjunction-like
otherwise.  A crucial point, which we learned when studying multiparty
session types~\cite{MTLC-g-sessiontype}, is that interpretations
should be based on sets of roles rather than just individual roles. In
other words, one side is allowed to play multiple roles
simultaneously. A sequent $\Gamma$ in multirole logic is a multiset of
i-formulas, and such a sequent is inherently many-sided as each $R$
appearing in $\Gamma$ represents precisely one side. As can be readily
expected, the cut-rule in (either conjunctive or disjunctive)
multirole logic is of the following form:
$$
\begin%
{array}{l}
\infer%
{\Gamma}
{\Gamma, \interp{R}{A} & \Gamma, \interp{\setcomp{R}}{A}}
\end{array}
$$
The cut-rule can be interpreted as some sort of communication between
two parties in distributed
programming~\cite{Abramsky94,BellinS94,CairesP10,Wadler:2012ua}. For
communication between multiple parties, it is natural to seek a
generalization of the cut-rule that involves more than two sequents.
In conjunctive multirole logic, the admissibility of the following
rule of the name $\ncutconj{n}$ can be established:
$$
\begin%
{array}{l}
\infer%
{\tpjg\Gamma}
{\setcomp{R_1}\uplus\cdots\uplus\setcomp{R_n}=\fullset &
 \tpjg\Gamma,\interp{R_1}{A} & \cdots & \tpjg\Gamma,\interp{R_n}{A}
}
\end{array}
$$
In disjunctive multirole logic, the admissibility of the following
rule of the name $\ncutdisj{n}$ can be established:
$$
\begin%
{array}{l}
\infer%
{\tpjg\Gamma}
{R_1\uplus\cdots\uplus R_n=\fullset &
 \tpjg\Gamma,\interp{R_1}{A} & \cdots & \tpjg\Gamma,\interp{R_n}{A}
}
\end{array}
$$
We may use the name $\ncut{n}$ to refer to either $\ncutconj{n}$
or $\ncutdisj{n}$.

\def\lneg{\lnot}
In classical logic, the negation operator is clearly one of a
kind. With respect to negation, the conjunction and disjunction
operators behave dually, and the universal and existential quantifiers
behave dually as well.  For the moment, let us write $\lneg{A}$ for
the negation of $A$.  It seems rather natural to interpret
$\interp{R}{\lneg{A}}$ as $\interp{\setcomp{R}}{A}$. Unfortunately,
such an interpretation of negation immediately breaks $n$-cut for any
$n\geq 3$. What we discover regarding negation is that the notion of
negation can be generalized to endmorphisms on the underlying set
$\fullset$ of roles.



\def%
\mrln#1#2%
{{\lneg_#1}(#2)}
\def%
\mrlc#1#2#3%
{#2\,{\land_#1}\,#3}
\def%
\mrlq#1#2%
{{\forall_#1}(#2)}
\def%
\mrlimp#1#2#3#4%
{#3\,{\limp_{#1,#2}{#4}}}

\def\F{{\cal F}}
\def\UF{{\cal U}}

\def\preimg#1#2{{#1}^{-1}(#2)}

\begin%
{figure}
\[
\begin%
{array}{c}
\infer%
[\hbox to 0pt{\bf(Id)\hss}]
{\tpjg\Gamma,\interp{R_1}{\prim},\ldots,\interp{R_n}{\prim}}
{\fullset=R_1\uplus\ldots\uplus R_n}
\\[6pt]
\infer%
[\hbox to 0pt{\bf(Weaken)\hss}]
{\tpjg\Gamma,\interp{R}{A}}
{\tpjg\Gamma,\interp{R}{A},\interp{R}{A}}
\\[6pt]
\infer%
[\hbox to 0pt{\bf($\lneg$)\hss}]
{\tpjg\Gamma,\interp{R}{\mrln{\fneg}{A}}}
{\tpjg\Gamma,\interp{\preimg{\fneg}{R}}{A}}
\\[6pt]
\infer%
[\hbox to 0pt{\bf($\land$-neg-l)\hss}]
{\tpjg\Gamma,\interp{R}{\mrlc{\UF}{A}{B}}}
{R\not\in\UF & \tpjg\Gamma,\interp{R}{A}}
\\[6pt]
\infer%
[\hbox to 0pt{\bf($\land$-neg-r)\hss}]
{\tpjg\Gamma,\interp{R}{\mrlc{\UF}{A}{B}}}
{R\not\in\UF & \tpjg\Gamma,\interp{R}{B}}
\\[6pt]
\infer%
[\hbox to 0pt{\bf($\land$-pos)\hss}]
{\tpjg\Gamma,\interp{R}{\mrlc{\UF}{A}{B}}}
{
R\in\UF &
\tpjg\Gamma,\interp{R}{A} & \tpjg\Gamma,\interp{R}{B}
}
\\[6pt]
\infer%
[\hbox to 0pt{\bf($\limp$-neg)\hss}]
{\tpjg\Gamma,\interp{R}{\mrlimp{\fneg}{\UF}{A}{B}}}
{R\not\in\UF & \tpjg\Gamma,\interp{\preimg{\fneg}{R}}{A},\interp{R}{B}}
\\[6pt]
\infer%
[\hbox to 0pt{\bf($\limp$-pos)\hss}]
{\tpjg\Gamma_1,\Gamma_2,\interp{R}{\mrlimp{\fneg}{\UF}{A}{B}}}
{R\in\UF &
  \tpjg\Gamma_1,\interp{\preimg{\fneg}{R}}{A} & \tpjg\Gamma_2,\interp{R}{B}}
\\[6pt]
\infer%
[\hbox to 0pt{\bf($\forall$-neg)\hss}]
{\tpjg\Gamma,\interp{R}{\mrlq{\UF}{\lambda x.A}}}
{R\not\in\UF & \tpjg\Gamma,\interp{R}{\subst{x}{t}{A}}}
\\[6pt]
\infer%
[\hbox to 0pt{\bf($\forall$-pos)\hss}]
{\tpjg\Gamma,\interp{R}{\mrlq{\UF}{\lambda x.A}}}
{R\in\UF & x\not\in\Gamma & \tpjg\Gamma,\interp{R}{A}}
\\[6pt]
\end{array}
\]
\label{figure:MRL:infrules}
\caption{The inference rules for MRL}
\end{figure}
\section%
{Multirole Logic}
\label{section:MRLogic}
Let $\fullset$ be the underlying set of roles for the multirole logic
MRL presented in this section. Strictly speaking, this MRL should be
referred to as {\em first-order predicate multirole logic\/}.

We use $t$ for first-order terms, which are standard (and thus not
formulated explicitly for brevity).

\begin%
{definition}
A filter $\F$ on $\fullset$ is a subset of the power set of $\fullset$
such that
\begin%
{itemize}
\item $\fullset\in\F$
\item $R_1\in\F$ and $R_1\subseteq R_2$ implies $R_2\in\F$
\item $R_1\in\F$ and $R_2\in\F$ implies $R_1\cap R_2\in\F$
\end{itemize}
A filter on $\fullset$ is an ultrafilter if either $R\in\F$ or
$\setcomp{R}\in\F$ holds for every subset $R$ of $\fullset$.  We use
$\UF$ to range over ultrafilters on $\fullset$. Note that each $\UF$
on $\fullset$ is of the form $\{R\subseteq\fullset\mid r\in R\}$ for
some $r\in\fullset$ if $\fullset$ is finite.
\end{definition}
Given an endmorphism $f$ on $\fullset$, we use $\lneg_{\fneg}$ for a
unary negative connective.  Given an ultrafilter $\UF$ on $\fullset$,
we $\land_{\UF}$ for a binary conjunctive connective and
$\forall_{\UF}$ for a universal quantifier. Given an endomorphism $f$
and an ultrafilter on $\fullset$, we use $\limp_{\fneg,\UF}$ for a
binary implicative connective. The formulas in $\MRL$ are defined as
follows:
\[
\begin%
{array}
{lrcl}
\mbox{formulas} & A & ::= & %
\prim \mid \mrln{\fneg}{A} \mid \mrlc{\UF}{A_1}{A_2} \mid \mrlimp{\fneg}{\UF}{A}{B} \mid \mrlq{\UF}{\lambda x.A} \\
\end{array}
\]
where $p$ ranges over pre-defined primitive formulas. Instead of
writing something like $\forall_{\UF} x.A$, we write
$\forall_{\UF}(\lambda x.A)$, where $x$ is a bound variable. Given a
formula $A$, a term $t$ and a variable $x$, we use $\subst{x}{t}{A}$
for the result of substituting $t$ for $x$ in $A$ and treat it as a
proper subformula of $\forall_{\UF}(\lambda x.A)$.

Given a formula $A$ and a set $R$ of roles, $\interp{R}{A}$ is
referred to as an i-formula (for interpretation of $A$ based on
$R$). Let us use $\Gamma$ for multisets of i-formulas, which are also
referred to as sequents.  The inference rules for $\MRL$ are listed
in Figure~\ref{figure:MRL:infrules}. In the rule
$\mbox{\bf($\forall$-pos)}$, $x\not\in\Gamma$ means that $x$ does not
have any free occurrences in i-formulas contained inside $\Gamma$.
Please note that a sequent $\Gamma$ in this formulation is many-sided
(rather than one-sided) as every $R$ appearing $\Gamma$ represents
precisely one side.


Let us use $\D$ for derivations of sequents, which are just trees
containing nodes that are applications of inference rules.  Given a
derivation $\D$, $\height{\D}$ stands for the tree height of $\D$.
When writing $\D::\Gamma$, we mean that $\D$ is a derivation of
$\Gamma$, that is, $\Gamma$ is the conclusion of $\D$. We may also use
the following format to present an inference rule:
$$(\Gamma_1;\ldots;\Gamma_n)\jimp\Gamma_0$$ where $\Gamma_i$ for
$1\leq i\leq n$ are the premisses of the rule and $\Gamma_0$ the
conclusion.

\begin%
{lemma}
(Weakening)
The following rule is admissible:
$$
(\Gamma)\jimp{\Gamma,\interp{R}{A}}
$$
\end{lemma}
\begin%
{proof}
By structural induction on the derivation of $\Gamma$.
\hfill\end{proof}

\begin%
{lemma}
\label{lemma:MRL:fullset}
The following rule is admissible:
$$
()\jimp{\Gamma,\interp{\fullset}{A}}
$$
\end{lemma}
\begin%
{proof}
By structural induction on $A$.
\hfill\end{proof}

\begin%
{lemma}
(1-cut)
\label{lemma:MRL:emptyset}
The following rule is admissible:
$$
(\Gamma,\interp{\emptyset}{A})\jimp
{\Gamma}
$$
\end{lemma}
\begin%
{proof}
Assume
$\D::(\Gamma,\interp{\emptyset}{A})$.
The proof proceeds by structural induction on $\D$.
\hfill\end{proof}
Note that Lemma~\ref{lemma:MRL:emptyset} can be seen as a special form
of cut-elimination where only one sequent is involved.


\begin%
{lemma}
The following rule is admissible:
$$
()\jimp{\Gamma,\interp{R}{A},\interp{\setcomp{R}}{A}}
$$
\end{lemma}
\begin%
{proof}
A proof for the lemma can be given based on structural induction on
$A$ directly. Also, the lemma immediately follows from
Lemma~\ref{lemma:MRL:fullset} and Lemma~\ref{lemma:MRL:role-split}
\hfill\end{proof}

\begin%
{lemma}
\label{lemma:MRL:cut-2-spill}
(2-cut with spill)
Assume that
$\setcomp{R_1}$
and
$\setcomp{R_2}$
are disjoint.
Then the following rule is admissible in $\MRL$:
$$
(\Gamma_1,\interp{R_1}{A};
 \Gamma_2,\interp{R_2}{A})\jimp
{\Gamma_1,\Gamma_2,\interp{R_1\cap R_2}{A}}
$$
\end{lemma}
\begin%
{proof}
Assume that we have
$\D_1::(\Gamma,\interp{R_1}{A})$
and
$\D_2::(\Gamma,\interp{R_2}{A})$.
The proof proceeds by induction
on the structure of $A$ and $\height{\D_1}+\height{\D_2}$,
lexicographically ordered.
\hfill\end{proof}

\begin%
{lemma}
(Splitting of Roles)
\label{lemma:MRL:role-split}
The following rule is admissible in MRL:
$$
(\Gamma,\interp{R_1\uplus R_2}{A})
\jimp
{\Gamma,\interp{R_1}{A},\interp{R_2}{A}}
$$
\end{lemma}
\begin%
{proof}
Assume that
$\D$ is a derivation of
$(\Gamma,\interp{R_1\uplus R_2}{A})$.
The proof proceeds by induction
on the structure of $A$ and $\height{\D}$, lexicographically ordered.
\hfill\end{proof}

\begin%
{lemma}
(mp-cut)
\label{lemma:MRL:multiparty-cut}
Assume that $R_1,\ldots,R_n$ are subsets of $\RC$ for some $n\geq 1$.
If $\setcomp{R}_1\uplus\cdots\uplus\setcomp{R}_n=\fullset$ holds, then the
following rule is admissible:
$$
(\Gamma_1,\interp{R_1}{A};\ldots;\Gamma_n,\interp{R_n}{A})\jimp(\Gamma_1,\ldots,\Gamma_n)
$$
\end{lemma}
\begin%
{proof}
The proof proceeds by induction on $n$.
If $n=1$,
then this lemma is just
Lemma~\ref{lemma:MRL:emptyset}.
Assume that $n\geq 2$ holds. Then we have
$\D_i::(\Gamma_i,\interp{R_i}{A})$ for $1\leq i\leq n$.
Clearly, $\setcomp{R_1}$ and $\setcomp{R_2}$ are disjoint.
By Lemma~\ref{lemma:MRL:cut-2-spill},
we have $\D_{12}::(\Gamma_1,\Gamma_2,\interp{{R_1\cap R_2}}{A})$.
By induction hypothesis, we can derive the sequent $\Gamma_1,\Gamma_2,\ldots,\Gamma_n$
based on $\D_{12},\ldots,\D_n$.
\hfill\end{proof}

This given proof of Lemma~\ref{lemma:MRL:multiparty-cut} clearly
indicates that multiparty cut-elimination builds on top of
Lemma~\ref{lemma:MRL:emptyset} and Lemma~\ref{lemma:MRL:cut-2-spill}.
In particular, one may see Lemma~\ref{lemma:MRL:emptyset} and
Lemma~\ref{lemma:MRL:cut-2-spill} as two fundamental meta-properties
of a logic.

\def\bang{{!}}
\def\qmark{{?}}
\def\lmrlx#1#2{(\bang{#2})_{#1}}
\begin%
{figure}
\[
\begin%
{array}{c}
\infer%
[\hbox to 0pt{\bf(Id)\hss}]
{\tpjg\interp{R_1}{\prim},\ldots,\interp{R_n}{\prim}}
{\fullset=R_1\uplus\ldots\uplus R_n}
\\[6pt]
\infer%
[\hbox to 0pt{\bf($\lneg$)\hss}]
{\tpjg\Gamma,\interp{R}{\mrln{\fneg}{A}}}
{\tpjg\Gamma,\interp{\preimg{\fneg}{R}}{A}}
\\[6pt]
\infer%
[\hbox to 0pt{\bf($\limp$-neg)\hss}]
{\tpjg\Gamma,\interp{R}{\mrlimp{\fneg}{\UF}{A}{B})}}
{R\not\in\UF & \tpjg\Gamma,\interp{\preimg{\fneg}{R}}{A},\interp{R}{B}}
\\[6pt]
\infer%
[\hbox to 0pt{\bf($\limp$-pos)\hss}]
{\tpjg\Gamma_1,\Gamma_2,\interp{R}{\mrlimp{\fneg}{\UF}{A}{B}}}
{
R\in\UF &
\tpjg\Gamma_1,\interp{\preimg{\fneg}{R}}{A} & \tpjg\Gamma_2,\interp{R}{B}
}
\\[6pt]
\infer%
[\hbox to 0pt{\bf($\land$-neg-l)\hss}]
{
\tpjg\Gamma,\interp{R}{\mrlc{\UF}{A}{B}}
}
{R\not\in\UF & \tpjg\Gamma,\interp{R}{A}}
\\[6pt]
\infer%
[\hbox to 0pt{\bf($\land$-neg-r)\hss}]
{
\tpjg\Gamma,\interp{R}{\mrlc{\UF}{A}{B}}
}
{R\not\in\UF & \tpjg\Gamma,\interp{R}{B}}
\\[6pt]
\infer%
[\hbox to 0pt{\bf($\land$-pos)\hss}]
{\tpjg\Gamma,\interp{R}{\mrlc{\UF}{A}{B}}}
{
R\in\UF &
\tpjg\Gamma,\interp{R}{A} & \tpjg\Gamma,\interp{R}{B}
}
\\[6pt]
\infer%
[\hbox to 0pt{\bf($\bang$-pos)\hss}]
{\tpjg\qmark(\Gamma)
,\interp{R}{\lmrlx{\UF}{A}}}
{R\in\UF & \tpjg\qmark(\Gamma),\interp{R}{A}}
\\[6pt]
\infer%
[\hbox to 0pt{\bf($\bang$-neg-weaken)\hss}]
{\tpjg\Gamma,\interp{R}{\lmrlx{\UF}{A}}}
{R\not\in\UF & \tpjg\Gamma}
\\[6pt]
\infer%
[\hbox to 0pt{\bf($\bang$-neg-derelict)\hss}]
{\tpjg\Gamma,\interp{R}{\lmrlx{\UF}{A}}}
{R\not\in\UF & \tpjg\Gamma,\interp{R}{A}}
\\[6pt]
\infer%
[\hbox to 0pt{\bf($\bang$-neg-contract)\hss}]
{\tpjg\Gamma,\interp{R}{\lmrlx{\UF}{A}}}
{R\not\in\UF & \tpjg\Gamma,\interp{R}{\lmrlx{\UF}{A}},\interp{R}{\lmrlx{\UF}{A}}}
\\[6pt]
\infer%
[\hbox to 0pt{\bf($\forall$-neg)\hss}]
{\tpjg\Gamma,\interp{R}{\mrlq{\UF}{\lambda x.A}}}
{R\not\in\UF & \tpjg\Gamma,\interp{R}{\subst{x}{t}{A}}}
\\[6pt]
\infer%
[\hbox to 0pt{\bf($\forall$-pos)\hss}]
{\tpjg\Gamma,\interp{R}{\mrlq{\UF}{\lambda x.A}}}
{R\in\UF & x\not\in\Gamma & \tpjg\Gamma,\interp{R}{A}}
\\[6pt]
\end{array}
\]
\label{figure:LMRL:infrules}
\caption{The inference rules for LMRL}
\end{figure}
\section%
{Linear Multirole Logic}
\label{section:LMRLogic}

In this section, we generalize classical linear logic (CLL) to linear
multirole logic (LMRL). The formulas in $\LMRL$ is defined as follows:
\[
\begin%
{array}
{lrcl}
\mbox{formulas} & A & ::= & %
\prim \mid \mrln{\fneg}{A} \mid \mrlc{\UF}{A_1}{A_2} \mid \mrlimp{\fneg}{\UF}{A}{B} \mid \lmrlx{\UF}{A} \mid \mrlq{\UF}{\lambda x.A} \\
\end{array}
\]
Let us write $A\otimes_{\UF} B$ as a shorthand for
$\mrlimp{id}{\UF}{A}{B}$, where $id$ stands for the identity function
on $\fullset$.  If one likes, one may also prefer to write
$A{\&}_{\UF} B$ for $A{\land}_{\UF} B$.  The inference rules for
$\LMRL$ are listed in Figure~\ref{figure:LMRL:infrules}.

\begin%
{lemma}
The following rule is admissible:
$$
()\jimp{\interp{\fullset}{A}}
$$
\end{lemma}
\begin%
{proof}
By structural induction on $A$. Note that we only need
positive rules to construct a proof of $\interp{\fullset}{A}$.
\hfill\end{proof}

\begin%
{lemma}
\label{lemma:LMRL:emptyset}
The following rule is admissible:
$$
(\Gamma,\interp{\emptyset}{A})\jimp
{\Gamma}
$$
\end{lemma}
\begin%
{proof}
Assume $\D::(\Gamma,\interp{\emptyset}{A})$.
We prove by induction on the height of $\D$
the existence of $\D'::(\Gamma,\interp{\emptyset}{A})$
such that $\height{\D'}\leq\height{\D}$ holds.

If $\D$ consists of an application of the axiom,
then the case is trivial.
If $\interp{\emptyset}{A}$ is introduced by
the last applied rule in $\D$, then the rule must be negative
and the case follows from the induction hypothesis on the immediate
subderivation of $\D$.
If $\interp{\emptyset}{A}$ is not introduced by
the last applied rule in $\D$, then the case is straightforward.
\hfill\end{proof}


\begin%
{lemma}
\label{lemma:LMRL:cut-2-spill}
(2-cut with spill)
Assume that
$\setcomp{R_1}$
and
$\setcomp{R_2}$
are disjoint.
Then the following rule is admissible in LMRL:
$$
(\Gamma_1,\interp{R_1}{A};\Gamma_2,\interp{R_2}{A})\jimp
{\Gamma_1,\Gamma_2,\interp{R_1\cap R_2}{A}}
$$
\end{lemma}

\begin%
{lemma}
(Splitting of Roles)
\label{lemma:LMRL:role-split}
The following rule is admissible in LMRL:
$$
(\Gamma,\interp{R_1\uplus R_2}{A})
\jimp
{\Gamma,\interp{R_1}{A},\interp{R_2}{A}}
$$
\end{lemma}

\begin%
{lemma}
(mp-cut)
\label{lemma:LMRL:multiparty-cut}
Assume that $R_1,\ldots,R_n$ are subsets of $\RC$ for some $n\geq 1$.
If $\setcomp{R}_1\uplus\cdots\uplus\setcomp{R}_n=\fullset$ holds, then the
following rule is admissible:
$$
(\Gamma_1,\interp{R_1}{A};\ldots;\Gamma_n,\interp{R_n}{A})\jimp\Gamma_1,\ldots,\Gamma_n
$$
\end{lemma}
\begin%
{proof}
The proof follows induction on $n$.
It is essentially
parallel to the proof of Lemma~\ref{lemma:MRL:multiparty-cut}.
\hfill\end{proof}


\section%
{Filter-Based Interpretation for Intuitionism}
We can introduce another parameter in MRL to account for
intuitionism, supporting a genuine unification of classical
logic and intuitionistic logic.

\begin%
{definition}%
Given a filter $\F$ on $\fullset$, a sequent $\Gamma$ is
$\F$-intuitionistic if there exists at most one i-formula
$\interp{R}{A}$ in $\Gamma$ such that $R\in\F$ holds.  An inference
rule is $\F$-intuitionistic if its conclusion is a $\F$-intuitionistic
sequent.
\end{definition}

\begin%
{definition}
(Intuitionistic MRL)
Given an ideal $\F$ on $\fullset$,
the inference rules in $\MRL_{\F}$ are those in $\MRL$ that are
$\F$-intuitionistic. We may refer to $\MRL_{\F}$ as the
$\F$-intuitionistic multirole logic.
\end{definition}
It can be readily noted that $\MRL$ is essentially equivalent to the
$\MRL_{\F}$ for $\F=\{\fullset\}$.

\begin%
{lemma}
\label{lemma:MRLF:construct}
Let $\F$ be a filter $\F$ on $\fullset$.
\begin%
{itemize}
\item
If $\tpjg\interp{R}{\mrlc{\UF}{A_1}{A_2}}$ is derivable in $\MRL_{\F}$
for some $R\not\in\UF$, then $\tpjg\interp{R}{A_i}$ is derivable in
$\MRL_{\F}$ for either $i=0$ or $i=1$.
\item
If $\tpjg\interp{R}{\mrlq{\UF}{\lambda x.A}}$ is derivable in
$\MRL_{\F}$ for some $R\not\in\UF$, then there exists a term $t$ such
that $\tpjg\interp{R}{\subst{x}{t}{A}}$ is derivable in $\MRL_{\F}$.
\end{itemize}
\end{lemma}

\begin%
{lemma}
\label{lemma:MRLF:multiparty-cut}
Given any filter $\F$ on $\fullset$, $\MRL_{\F}$ enjoys multiparty
cut-elimination.
\end{lemma}

Similarly, $\F$-intuitionistic $\LMRL$ ($\LMRL_{\F}$) can be defined,
and both Lemma~\ref{lemma:MRLF:construct} and and
Lemma~\ref{lemma:MRLF:multiparty-cut} have obvious corresponding
versions that hold for $\LMRL_{\F}$.

\section%
{Related Work and Conclusion}
\label{section:relatedwork_conclusion}
The first and foremost inspiration for multirole logic came from a
study on multiparty session types in distributed
programming~\cite{MTLC-i-sessiontype,MTLC-g-sessiontype}, which was in
turn closely related to series of earlier
work~\cite{Abramsky94,BellinS94,CairesP10,Wadler:2012ua}. Also, MCP, a
variant of CLL that admits a generalized cut-rule for composing
multiple proofs, is first introduced in a paper by Carbone et
al~\cite{Carbone:2015hl}. In the following work~\cite{Carbone:2016kd},
a variant of MCP is introduced, and a translation from MCP to
CP~\cite{Wadler:2012ua} via GCP (an intermediate calculus) is given
that interprets a coherence proof in MCP as an arbiter process for
mediating communications in a multiparty session.

For long, studies on logics have been greatly influencing research on
programming languages. In the case of multirole logic, we see a
genuine example that demonstrates the influence of the latter on the
former. What an influence it is! If just for one thing only, we should
immediately revisit some classical results in logic and
recast/reinterpret them in the framework of multirole logic.
\bibliographystyle{jfp}\bibliography{mybib}

\end{document}